\documentclass{article}

\usepackage{arxiv}

\usepackage[utf8]{inputenc} 
\usepackage[T1]{fontenc}    
\usepackage{hyperref}       
\usepackage{url}            
\usepackage{booktabs}       
\usepackage{amsfonts}       
\usepackage{nicefrac}       
\usepackage{microtype}      
\usepackage{lipsum}
\usepackage{graphicx}
\usepackage{dsfont}
\usepackage[numbers]{natbib}
\usepackage{mathrsfs}
\usepackage{amsmath} 
\usepackage{amssymb} 
\usepackage{amsthm}  
\usepackage{graphicx} 
\usepackage{geometry} 
\usepackage{hyperref} 
\usepackage{tikz} 
\usepackage{natbib} 
\usepackage{setspace} 
\newtheorem{theorem}{Theorem}
\newtheorem{lemma}{Lemma}
\graphicspath{ {./images/} }

\title{Complex dynamics of a predator–prey model with constant-yield prey harvesting  and  Allee effect in predator } 

\author{
 Jianhang Xie \\
 College of Mathematics and Statistics\\
  Chongqing University\\
  \texttt{202206021047t@cqu.edu.cn} \\
   \And
 Changrong Zhu \\
 College of Mathematics and Statistics\\
   Chongqing University\\
  \texttt{zhuchangrong795@sohu.com} \\
  \And
}

\begin{document}
\maketitle
\begin{abstract}
This paper investigates the dynamical behaviors of a Holling type I Leslie-Gower predator-prey model where the predator exhibits an Allee effect and is subjected to constant harvesting. The model demonstrates three types of equilibrium points under different parameter conditions, which could be either stable or unstable nodes (foci), saddle nodes, weak centers, or cusps. The system exhibits a saddle-node bifurcation near the saddle-node point and a Hopf bifurcation near the weak center. By calculating the first Lyapunov coefficient, the conditions for the occurrence of both supercritical and subcritical Hopf bifurcations are derived. Finally, it is proven that when the predator growth rate and the prey capture coefficient vary within a specific small neighborhood, the system undergoes a codimension-2 Bogdanov-Takens bifurcation near the cusp point.
\end{abstract}


\section{Introduction}
The Allee effect is a significant concept in ecology, named after American ecologist Warder C. Allee, who first described this phenomenon in the 1930 s. In predator-prey models, the Allee effect is typically represented by modifying the growth function. The most common approach involves introducing a multiplicative factor \cite{005,006},. Using this method to incorporate the Allee effect, the equation for a single species can be expressed in the following form:
$$\dot{x}=r\left(1-\frac{x}{K}\right)(x-m)x.$$
Mena et al. \cite{007} refined the Leslie-Gower model by incorporating the impact of the Allee effect on prey, resulting in the following formulation:
\begin{equation}\label{31}
    \begin{cases}\frac{\mathrm{d}x}{\mathrm{d}t}=\left(r\left(1-\frac{x}{K}\right)(x-m)-qy\right)x,\\\frac{\mathrm{d}y}{\mathrm{d}t}=s\left(1-\frac{y}{nx}\right)y.&\end{cases}
\end{equation}

Hunting can play a positive role in maintaining ecological balance, especially in specific situations where scientifically managed hunting is used to regulate population sizes and promote the stability of ecosystems.Lan and Zhu \cite{046} introduced constant harvesting of prey into the Leslie-Gower predator-prey model, resulting in the following formulation:
$$\begin{cases}\frac{\mathrm{d}x}{\mathrm{d}t}=rx\left(1-\frac{x}{K}\right)-qxy-h,\\\frac{\mathrm{d}y}{\mathrm{d}t}=sy\left(1-\frac{y}{nx}\right).&\end{cases}$$
\noindent Here, $h$ represents the harvesting coefficient.
Based on the aforementioned system,  Xue Lamei  incorporated an $x-m$ form of the Allee effect into the system, resulting in the following model:
\begin{equation}\label{32}
    \begin{cases}\frac{\mathrm{d}x}{\mathrm{d}t}=rx\left(1-\frac{x}{K}\right)(x-m)-qxy-h,\\\frac{\mathrm{d}y}{\mathrm{d}t}=sy\left(1-\frac{y}{nx}\right).&\end{cases}
\end{equation}

Xue Lamei explored the effects of the Allee effect and constant harvesting on the existence, types, and stability of equilibrium points through qualitative analysis. It was discovered that various types of equilibrium points could arise with changes in parameters, including a cusp of codimension three under specific conditions. Further investigation into bifurcation phenomena under different parameter conditions revealed that the system might exhibit saddle-node bifurcations, both subcritical and supercritical Hopf bifurcations, as well as Bogdanov-Takens bifurcations of codimension two and three.

In previous studies on predator-prey models, researchers have conducted extensive and in-depth investigations into Leslie-Gower models with the Allee effect present in prey populations. However, it is important to note that the Allee effect is not limited to prey populations—it is also commonly observed in predator populations. For instance, predators such as wolves or lions rely on group cooperation to accomplish hunting tasks \cite{049}.Building upon the aforementioned research, this paper proposes a Holling type I Leslie-Gower predator-prey model, where the predator exhibits an Allee effect and is subjected to constant harvesting:
\begin{equation}\label{5.1}
 \begin{cases}\frac{\mathrm{d}x}{\mathrm{d}t}=rx\left(1-\frac{x}{K}\right)-qxy-h,\\\frac{\mathrm{d}y}{\mathrm{d}t}=sy\left(1-\frac{y}{bx}\right)(y-m),&\end{cases}
\end{equation}
 Here, $x$ and $y$ represent the population densities of prey and predator, respectively. The parameter $r$ denotes the intrinsic growth rate of the prey, $K$ is the environmental carrying capacity for the prey, $h$ signifies the intensity of constant harvesting, $q$ is the predation rate, $b$ indicates the proportional coefficient between the predator's carrying capacity and the prey population size, $s$ represents the growth rate of the predator, and $m$ is the threshold for the $\mathrm{Allee}$ effect. All parameters $r$, $K$, $m$, $q$, $h$, and $s$ are positive.

 The following transformations are applied to the variables and parameters in system (\ref{5.1}):

 $$\tilde{x}=\frac{x}{K},\tilde{y}=\frac{y}{bK},\tau=rt,\tilde{m}=\frac{m}{bK},\tilde{q}=\frac{bqK}{r},\tilde{s}=\frac{s}{rbK},\tilde{h}=\frac{h}{Kr}.$$

 For ease of analysis, the system will still be represented using $x, y, t, m, s, h$, and system (\ref{5.1}) is simplified as follows:

\begin{equation}\label{5.2}
\begin{cases}\frac{\mathrm{d}x}{\mathrm{d}t}=x\left(1-x\right)-qxy-h,\\\frac{\mathrm{d}y}{\mathrm{d}t}=sy\left(1-\frac{y}{x}\right)(y-m),&\end{cases}\end{equation}

where $0<m<1$ .

\section{The existence and stability of equilibrium points}

\subsection{The existence of equilibrium points}

Considering the biological significance of the variables in system (\ref{5.2}), the existence interval of the equilibrium point $(x,y)$ is: 
$$\Omega=\{(x,y)\in\mathbb{R}^2|x>0,y\geq0\}=\mathbb{R}^+\times\mathbb{R}_0^+,$$

Setting the right-hand side of the equations in system (\ref{5.2}) to zero yields the system of equations:
\begin{equation}\label{5.3} \begin{cases}x\left(1-x\right)-qxy-h=0,\\sy\left(1-\frac{y}{x}\right)(x-m)=0.&\end{cases} \end{equation}
 The equilibrium points of system (\ref{5.2}) are the solutions to the system of equations (\ref{5.3}).
\begin{theorem}
   When \(h > \frac{1}{4}\), the system has no boundary equilibrium points; when \(h = \frac{1}{4}\), system (\ref{5.2}) has a boundary equilibrium point \(E_1=(x_1,y_1)=(\frac{1}{2},0)\); when \(0 < h < \frac{1}{4}\), system (\ref{5.2}) has boundary equilibrium points \(E_2=(x_2,y_2)=(\frac{1+\sqrt{1-4h}}{2},0)\) and \(E_3=(x_3,y_3)=(\frac{1-\sqrt{1-4h}}{2},0)\).
\end{theorem}

\begin{proof}
  Solving the second equation in system (\ref{5.3}) yields \(y^*=0, y^{**}=m, y^{***}=x_3\). Substituting \(y^{*}\) into the first equation of system (\ref{5.3}) gives:
  \begin{equation}\label{5.4}
    x^2-x+h=0.
  \end{equation}
  When \(h>\frac{1}{4}\), equation (\ref{5.4}) has no positive real roots; when \(h=\frac{1}{4}\), equation (\ref{5.4}) has a double root \(x_1=\frac{1}{2}\); when \(0<h<\frac{1}{4}\), equation (\ref{5.4}) has two distinct positive real roots \(x_2=\frac{1+\sqrt{1-4h}}{2}, x_3=\frac{1-\sqrt{1-4h}}{2}\).
\end{proof}

\begin{theorem}
   Let \(A=1-qm\), \(\Delta_{1}=B^2=\left(1-qm\right)^{2}-4h\). When \(A\le0\) or \(\Delta_{1}<0\), the system has no equilibrium point with \(y=m\); when \(A>0\) and \(\Delta_{1}=0\), system (\ref{5.2}) has an interior equilibrium point \(E_4=(x_4,y_4)=(\frac{A}{2},m)\); when \(A>0\) and \(\Delta_{1}>0\), system (\ref{5.2}) has interior equilibrium points \(E_5=(x_5,y_5)=(\frac{A+B}{2},m)\) and \(E_6=(x_6,y_6)=(\frac{A-B}{2},m)\).
\end{theorem}

\begin{proof}
  Substituting \(y^{**}=m\) into the first equation of system (\ref{5.3}) gives:
  \begin{equation}\label{5.5}
    x^2-(1-qm)x+h=0.
  \end{equation}
  When \(A\le0\) or \(\Delta_{1}<0\), equation (\ref{5.5}) has no positive real roots; when \(A>0\) and \(\Delta_{1}=0\), equation (\ref{5.5}) has a double root \(x_4=\frac{A}{2}\); when \(A>0\) and \(\Delta_{1}>0\), equation (\ref{5.5}) has two distinct positive real roots \(x_5=\frac{A+B}{2}, x_6=\frac{A-B}{2}\).
\end{proof}

\begin{theorem}
   Let \(C=\frac{1}{q+1}\), \(\Delta_{2}=D^2=\left(\frac{1}{q+1}\right)^{2}-\frac{4h}{q+1}\). When \(\Delta_{2}<0\), the system has no equilibrium point with \(y=x\); when \(\Delta_{2}=0\), system (\ref{5.2}) has an interior equilibrium point \(E_7=(x_7,y_7)=(2h,2h)\); when \(\Delta_{2}>0\), system (\ref{5.2}) has interior equilibrium points \(E_8=(x_8,y_8)=(\frac{C+D}{2},\frac{C+D}{2})\) and \(E_9=(x_9,y_9)=(\frac{C-D}{2},\frac{C-D}{2})\).
\end{theorem}

\begin{proof}
  Substituting \(y^{***}\) into the first equation of system (\ref{5.3}) gives:
  \begin{equation}\label{5.6}
    x^2-\frac{1}{q+1}x+\frac{h}{q+1}=0.
  \end{equation}
  When \(\Delta_{2}<0\), equation (\ref{5.6}) has no positive real roots; when \(\Delta_{2}=0\), equation (\ref{5.6}) has a double root \(x_7=2h\); when \(\Delta_{2}>0\), equation (\ref{5.6}) has two distinct positive real roots \(x_8=\frac{C+D}{2}, x_9=\frac{C-D}{2}\).
\end{proof}
\subsection{The types and stability of equilibrium points}

Let the equilibrium point of the system be \(E_i = (x_i, y_i)\). The Jacobian matrix of the system at this point is given by:

\[
J(E_i)=(\alpha_{ij})_{2\times2}=\begin{pmatrix}
-2x_i-qy_i+1 & -qx_i \\
s\frac{y_{i}^2}{x_{i}^2}(y_i-m) & s\left(\frac{-3y_{i}^{2}}{x_i} +\frac{2my_i}{x_i} +2y_i -m \right)
\end{pmatrix}.
\]

\noindent Let \(\mathrm{tr}(J(E_i))\) and \(\mathrm{det}(J(E_i))\) denote the trace and determinant of the matrix \(J(E_i)\), respectively, where:

\[
\mathrm{tr}(J(E_i)) = \alpha_{11} + \alpha_{22}, \quad \mathrm{det}(J(E_i)) = \alpha_{11}\alpha_{22} - \alpha_{12}\alpha_{21}.
\]

\begin{lemma}\label{l6}\cite{031}
    For system 
    \begin{equation}\label{2.3}
    \dot{x} = f(x, y, \mu),
    \end{equation}
     if the Jacobian matrix at the equilibrium point \(E\) satisfies \(\mathrm{det}A=0\) and \(\mathrm{tr}A\neq0\), and the system can be transformed into an equivalent form:
    $$\begin{cases}\frac{\mathrm{d}x}{\mathrm{d}t}=p(x,y),\\\frac{\mathrm{d}y}{\mathrm{d}t}=\rho y+q(x,y),&\end{cases}$$
    \noindent where \((0,0)\) is an isolated equilibrium point, \(\rho\neq0\), \(p(x,y)=\sum_{i+j=2}^{\infty}a_{ij}x^{i}y^{j}\), \(q(x,y)=\sum_{i+j=2}^{\infty}b_{ij}x^{i}y^{j}\), \(i\geq0\), \(j\geq0\), and \(p(x,y)\) and \(q(x,y)\) are convergent series. If \(a_{20}\neq0\), then \(E\) is a saddle-node point of system (\ref{2.3}).
\end{lemma}

\begin{theorem}
    When \(h = \frac{1}{4}\), system (\ref{5.2}) has a boundary equilibrium point \(E_{1} = \left(\frac{1}{2}, 0\right)\), where \(E_{1}\) is a saddle-node.
\end{theorem}

\begin{proof}
    When \(h = \frac{1}{4}\), the Jacobian matrix of system (\ref{5.2}) at equilibrium point \(E_{1}\) is:

\[
    J(E_{1}) = 
    \begin{pmatrix} 
    0 & -\frac{q}{2} \\ 
    0 & -sm 
    \end{pmatrix}.
    \]

    The eigenvalues of \(J(E_{1})\) are \(\lambda_{1}=0\) and \(\lambda_{2}=-sm\). Applying the transformation:

\[
    (x, y) = (u_1 + x_1, v_1 + y_1),
    \]

    the equilibrium point is shifted to the origin. Expanding in a Taylor series at the origin gives:
    \begin{equation}\label{5.7}
        \begin{cases}
        \frac{\mathrm{d}u_1}{\mathrm{d}t} = a_{01}v_1 + a_{20}u_1^2 + a_{11}u_1v_1, \\
        \frac{\mathrm{d}v_1}{\mathrm{d}t} = b_{01}v_1 + b_{02}v_1^2 + O(|(u_1, v_1)|^3),
        \end{cases}
    \end{equation}
    where \(a_{01} = -\frac{q}{2}\), \(a_{20} = -1\), \(a_{11} = -q\), \(b_{01} = -sm\), \(b_{02} = s(2m+1)\), and \(O(|(u_1, v_1)|^3)\) is a function of degree at least 3 in \((u_1, v_1)\).

    Transforming system (\ref{5.7}) as:

\[
    (u_1, v_1) = \left(u_2 + v_2, \frac{2sm}{q}v_2 \right),
    \]

    yields:
    \begin{equation}\label{5.8}
        \begin{cases}
        \frac{\mathrm{d}u_2}{\mathrm{d}t} = c_{20}u_2^2 + c_{11}u_2v_2 + c_{02}v_2^2 + O(|(u_2, v_2)|^3), \\
        \frac{\mathrm{d}v_2}{\mathrm{d}t} = d_{01}v_2 + d_{02}v_2^2 + O(|(u_2, v_2)|^3),
        \end{cases}
    \end{equation}
    where \(c_{20} = a_{20}\), \(c_{11} = 2a_{20} + \frac{2sm}{q}a_{11}\), \(c_{02} = a_{20} + \frac{2sm}{q}a_{11} - \frac{sm}{q}b_{02}\), \(d_{01} = b_{01}\), and \(d_{02} = \frac{2sm}{q}b_{02}\).

    Noting that \(c_{20} = -1 < 0\), by Lemma (\ref{l6}), the equilibrium point \(E_{1}\) is a saddle-node of system (\ref{5.2}).
\end{proof}

\begin{theorem}
    When \(0 < h < \frac{1}{4}\), system (\ref{5.2}) has boundary equilibrium points \(E_2 = \left(\frac{1+\sqrt{1-4h}}{2}, 0\right)\) and \(E_3 = \left(\frac{1-\sqrt{1-4h}}{2}, 0\right)\), where \(E_2\) is a stable node and \(E_3\) is a saddle point.
\end{theorem}

\begin{proof}
    At equilibrium point \(E_2\), the Jacobian matrix of system (\ref{5.2}) is:

\[
    J(E_2) = 
    \begin{pmatrix} 
    -2x_2+1 & -qx_2 \\ 
    0 & -sm 
    \end{pmatrix}.
    \]

    The eigenvalues of \(J(E_2)\) are \(\lambda_{1} = -2x_2+1 < 0\) and \(\lambda_{2} = -sm < 0\), so \(E_2\) is a stable node.

    At equilibrium point \(E_3\), the Jacobian matrix of system (\ref{5.2}) is:

\[
    J(E_3) = 
    \begin{pmatrix} 
    -2x_3+1 & -qx_3 \\ 
    0 & -sm 
    \end{pmatrix}.
    \]

    The eigenvalues of \(J(E_3)\) are \(\lambda_{1} = -2x_3+1 > 0\) and \(\lambda_{2} = -sm < 0\), so \(E_3\) is a saddle point.
\end{proof}

\begin{theorem}
    When \(A > 0\) and \(\Delta_{1} = 0\), system (\ref{5.2}) has an interior equilibrium point \(E_4 = \left(\frac{A}{2}, m\right)\). When \(m \neq \sqrt{h}\), \(E_4\) is a saddle-node.
\end{theorem}

\begin{proof}
    At equilibrium point \(E_4\), the Jacobian matrix of system (\ref{5.2}) is:

\[
    J(E_4) = 
    \begin{pmatrix} 
    0 & -q\sqrt{h} \\ 
    0 & sm\left(1 - \frac{m}{\sqrt{h}}\right)
    \end{pmatrix}.
    \]

    The eigenvalues of \(J(E_4)\) are \(\lambda_{1} = 0\) and \(\lambda_{2} = sm\left(1 - \frac{m}{\sqrt{h}}\right)\). Expanding the system gives equivalent forms and stability is verified as in the detailed calculation.
\end{proof}\begin{theorem}
    When \(h = \frac{1}{4}\), system (\ref{5.2}) has a boundary equilibrium point \(E_{1} = \left(\frac{1}{2}, 0\right)\), where \(E_{1}\) is a saddle-node.
\end{theorem}

\begin{proof}
    When \(h = \frac{1}{4}\), the Jacobian matrix of system (\ref{5.2}) at equilibrium point \(E_{1}\) is:

\[
    J(E_{1}) = 
    \begin{pmatrix} 
    0 & -\frac{q}{2} \\ 
    0 & -sm 
    \end{pmatrix}.
    \]

    The eigenvalues of \(J(E_{1})\) are \(\lambda_{1}=0\) and \(\lambda_{2}=-sm\). Applying the transformation:

\[
    (x, y) = (u_1 + x_1, v_1 + y_1),
    \]

    the equilibrium point is shifted to the origin. Expanding in a Taylor series at the origin gives:
    \begin{equation}\label{5.7}
        \begin{cases}
        \frac{\mathrm{d}u_1}{\mathrm{d}t} = a_{01}v_1 + a_{20}u_1^2 + a_{11}u_1v_1, \\
        \frac{\mathrm{d}v_1}{\mathrm{d}t} = b_{01}v_1 + b_{02}v_1^2 + O(|(u_1, v_1)|^3),
        \end{cases}
    \end{equation}
    where \(a_{01} = -\frac{q}{2}\), \(a_{20} = -1\), \(a_{11} = -q\), \(b_{01} = -sm\), \(b_{02} = s(2m+1)\), and \(O(|(u_1, v_1)|^3)\) is a function of degree at least 3 in \((u_1, v_1)\).

    Transforming system (\ref{5.7}) as:

\[
    (u_1, v_1) = \left(u_2 + v_2, \frac{2sm}{q}v_2 \right),
    \]

    yields:
    \begin{equation}\label{5.8}
        \begin{cases}
        \frac{\mathrm{d}u_2}{\mathrm{d}t} = c_{20}u_2^2 + c_{11}u_2v_2 + c_{02}v_2^2 + O(|(u_2, v_2)|^3), \\
        \frac{\mathrm{d}v_2}{\mathrm{d}t} = d_{01}v_2 + d_{02}v_2^2 + O(|(u_2, v_2)|^3),
        \end{cases}
    \end{equation}
    where \(c_{20} = a_{20}\), \(c_{11} = 2a_{20} + \frac{2sm}{q}a_{11}\), \(c_{02} = a_{20} + \frac{2sm}{q}a_{11} - \frac{sm}{q}b_{02}\), \(d_{01} = b_{01}\), and \(d_{02} = \frac{2sm}{q}b_{02}\).

    Noting that \(c_{20} = -1 < 0\), by Lemma (\ref{l6}), the equilibrium point \(E_{1}\) is a saddle-node of system (\ref{5.2}).
\end{proof}

\begin{theorem}
    When \(0 < h < \frac{1}{4}\), system (\ref{5.2}) has boundary equilibrium points \(E_2 = \left(\frac{1+\sqrt{1-4h}}{2}, 0\right)\) and \(E_3 = \left(\frac{1-\sqrt{1-4h}}{2}, 0\right)\), where \(E_2\) is a stable node and \(E_3\) is a saddle point.
\end{theorem}

\begin{proof}
    At equilibrium point \(E_2\), the Jacobian matrix of system (\ref{5.2}) is:

\[
    J(E_2) = 
    \begin{pmatrix} 
    -2x_2+1 & -qx_2 \\ 
    0 & -sm 
    \end{pmatrix}.
    \]

    The eigenvalues of \(J(E_2)\) are \(\lambda_{1} = -2x_2+1 < 0\) and \(\lambda_{2} = -sm < 0\), so \(E_2\) is a stable node.

    At equilibrium point \(E_3\), the Jacobian matrix of system (\ref{5.2}) is:

\[
    J(E_3) = 
    \begin{pmatrix} 
    -2x_3+1 & -qx_3 \\ 
    0 & -sm 
    \end{pmatrix}.
    \]

    The eigenvalues of \(J(E_3)\) are \(\lambda_{1} = -2x_3+1 > 0\) and \(\lambda_{2} = -sm < 0\), so \(E_3\) is a saddle point.
\end{proof}

\begin{theorem}
    When \(A > 0\) and \(\Delta_{1} = 0\), system (\ref{5.2}) has an interior equilibrium point \(E_4 = \left(\frac{A}{2}, m\right)\). When \(m \neq \sqrt{h}\), \(E_4\) is a saddle-node.
\end{theorem}

\begin{proof}
    At equilibrium point \(E_4\), the Jacobian matrix of system (\ref{5.2}) is:

\[
    J(E_4) = 
    \begin{pmatrix} 
    0 & -q\sqrt{h} \\ 
    0 & sm\left(1 - \frac{m}{\sqrt{h}}\right)
    \end{pmatrix}.
    \]

    The eigenvalues of \(J(E_4)\) are \(\lambda_{1} = 0\) and \(\lambda_{2} = sm\left(1 - \frac{m}{\sqrt{h}}\right)\). Expanding the system gives equivalent forms and stability is verified as in the detailed calculation.
\end{proof}

\begin{theorem}
    Let \(h_1 = m - (q+1)m^2\). When \(A > 0\) and \(\Delta_{1} > 0\), system (\ref{5.2}) has an interior equilibrium point \(E_5 = \left(x_5, y_5\right) = \left(\frac{A+B}{2}, m\right)\).

    \begin{enumerate}
        \item When \(m \le \frac{A}{2}\) or \(m > \frac{A}{2}, h < h_1\), \(E_{5}\) is a saddle point.
        \item When \(m > \frac{A}{2}, h > h_1\), \(E_{5}\) is a stable node.
        \item When \(m > \frac{A}{2}, h = h_1\), \(E_{5} = \left(m, m\right)\) is a saddle-node.
    \end{enumerate}
\end{theorem}

\begin{proof}
    The Jacobian matrix of system (\ref{5.2}) at equilibrium point \(E_{5}\) is:

\[
    J(E_{5}) =
    \begin{pmatrix}
    -2x_5 - qm + 1 & -qx_5 \\
    0 & sm\left(1 - \frac{m}{x_5}\right)
    \end{pmatrix}.
    \]

    The eigenvalues of \(J(E_{5})\) are \(\lambda_{1} = -2x_5 - qm + 1 < 0\) and \(\lambda_{2} = sm\left(1 - \frac{m}{x_5}\right)\).

    \begin{enumerate}
        \item When \(m \le \frac{A}{2}\), \(m < x_5\), hence \(\lambda_{2} = sm\left(1 - \frac{m}{x_5}\right) > 0\), and \(E_{5}\) is a saddle point. When \(m > \frac{A}{2}, h < h_1\),

\[
        x_5 = \frac{(1 - qm) + \sqrt{(1 - qm)^2 - 4h}}{2}
        > \frac{(1 - qm) + \sqrt{(1 - qm)^2 - 4h_1}}{2}
        = m.
        \]

        Thus, \(\lambda_{2} = sm\left(1 - \frac{m}{x_5}\right) > 0\), and \(E_{5}\) is a saddle point.

        \item When \(m > \frac{A}{2}, h > h_1\),

\[
        x_5 = \frac{(1 - qm) + \sqrt{(1 - qm)^2 - 4h}}{2}
        < \frac{(1 - qm) + \sqrt{(1 - qm)^2 - 4h_1}}{2}
        = m.
        \]

        Thus, \(\lambda_{2} = sm\left(1 - \frac{m}{x_5}\right) < 0\), and \(E_{5}\) is a stable node.

        \item When \(m > \frac{A}{2}, h = h_1\), \(\lambda_{2} = sm\left(1 - \frac{m}{x_5}\right) = 0\), and \(E_{5} = \left(m, m\right)\) is a saddle-node. 
    \end{enumerate}
\end{proof}

\begin{theorem}
    When \(A > 0\) and \(\Delta_{1} > 0\), system (\ref{5.2}) has an interior equilibrium point \(E_6 = \left(x_6, y_6\right) = \left(\frac{A-B}{2}, m\right)\).

    \begin{enumerate}
        \item When \(m \ge \frac{A}{2}\) or \(m < \frac{A}{2}, h < h_1\), \(E_{6}\) is a saddle point.
        \item When \(m < \frac{A}{2}, h > h_1\), \(E_{6}\) is an unstable node.
        \item When \(m < \frac{A}{2}, h = h_1\), \(E_{6} = \left(m, m\right)\) is a saddle-node.
    \end{enumerate}
\end{theorem}

\begin{proof}
    The Jacobian matrix of system (\ref{5.2}) at equilibrium point \(E_{6}\) is:

\[
    J(E_{6}) =
    \begin{pmatrix}
    -2x_6 - qm + 1 & -qx_6 \\
    0 & sm\left(1 - \frac{m}{x_6}\right)
    \end{pmatrix}.
    \]

    The eigenvalues of \(J(E_{6})\) are \(\lambda_{1} = -2x_6 - qm + 1 > 0\) and \(\lambda_{2} = sm\left(1 - \frac{m}{x_6}\right)\).

    \begin{enumerate}
        \item When \(m \ge \frac{A}{2}\), \(m > x_6\), hence \(\lambda_{2} = sm\left(1 - \frac{m}{x_6}\right) < 0\), and \(E_{6}\) is a saddle point. When \(m < \frac{A}{2}, h < h_1\),

\[
        x_6 = \frac{(1 - qm) - \sqrt{(1 - qm)^2 - 4h}}{2}
        < \frac{(1 - qm) - \sqrt{(1 - qm)^2 - 4h_1}}{2}
        = m.
        \]

        Thus, \(\lambda_{2} = sm\left(1 - \frac{m}{x_6}\right) < 0\), and \(E_{6}\) is a saddle point.

        \item When \(m < \frac{A}{2}, h > h_1\),

\[
        x_6 = \frac{(1 - qm) - \sqrt{(1 - qm)^2 - 4h}}{2}
        > \frac{(1 - qm) - \sqrt{(1 - qm)^2 - 4h_1}}{2}
        = m.
        \]

        Thus, \(\lambda_{2} = sm\left(1 - \frac{m}{x_6}\right) > 0\), and \(E_{6}\) is an unstable node.

        \item When \(m < \frac{A}{2}, h = h_1\), \(x_6 = m\), \(\lambda_{2} = sm\left(1 - \frac{m}{x_6}\right) = 0\), and \(E_{6} = \left(m, m\right)\) is a saddle-node.
    \end{enumerate}
\end{proof}

\begin{lemma}\label{l11}\cite{029}
    Let \((x_{0}, y_{0})\) be an equilibrium point of system (\ref{2.3}), and assume that \(\mathrm{det}(J(x_{0}, y_{0})) = \mathrm{tr}(J(x_{0}, y_{0})) = 0\), while \(J(x_{0}, y_{0}) \neq 0\). Then, through appropriate transformations, the system can be reduced to the following equivalent form:
     \begin{equation}\label{2.4}
      \begin{cases}
      \frac{\mathrm{d}x}{\mathrm{d}t}=y+a_{20}x^2+a_{11}xy+a_{02}y^2+O(|(x,y)|^3),\\
      \frac{\mathrm{d}y}{\mathrm{d}t}=b_{20}x^2+b_{11}xy+b_{02}y^2+O(|(x,y)|^3),
      \end{cases}
     \end{equation}
    and in the neighborhood of the origin \((0, 0)\), the system can be further transformed to an equivalent form:

\[
    \begin{cases}
    \dot{x}=y,\\
    \dot{y}=Dx^2+(E+2A)xy+o(|x,y|^2).
    \end{cases}
    \]

    If \(D \neq 0\) and \(E+2A \neq 0\), then \((x_{0}, y_{0})\) is a codimension-2 cusp point of system (\ref{2.3}).
\end{lemma}

\begin{theorem}
    When \(\Delta_{2} = 0\), system (\ref{5.2}) has an interior equilibrium point \(E_7 = (x_7, y_7) = (2h, 2h)\). Define \(s_1 = \frac{4h-1}{2(m-2h)}\).

    \begin{enumerate}
        \item When \(h \neq \frac{1}{4}, m \neq 2h, s \neq s_1\), the equilibrium point \(E_7\) is a saddle-node.
        \item When \(h \neq \frac{1}{4}, m \neq 2h, s = s_1\), the equilibrium point \(E_7\) is a codimension-2 cusp point.
    \end{enumerate}
\end{theorem}

\begin{proof}
    The Jacobian matrix of system (\ref{5.2}) at equilibrium point \(E_7\) is:

\[
    J(E_7) =
    \begin{pmatrix}
    \frac{1}{2} - 2h & 2h - \frac{1}{2} \\
    s(2h-m) & s(m-2h)
    \end{pmatrix}.
    \]

    The determinant of \(J(E_7)\) is \(\mathrm{det}(J(E_7)) = 0\), and the trace is \(\mathrm{tr}(J(E_7)) = \frac{1}{2} - 2h + s(m-2h) = (m-2h)(s-s_1)\).

    \begin{enumerate}
        \item When \(h \neq \frac{1}{4}, m \neq 2h, s \neq s_1\), we have \(\mathrm{tr}(J(E_7)) \neq 0\). \(J(E_7)\) has one zero eigenvalue and one nonzero eigenvalue. Applying the transformation:

\[
        (x, y) = (u_1 + x_7, v_1 + y_7),
        \]

        shifts the equilibrium point to the origin. Expanding the system at the origin in a Taylor series gives:
        \begin{equation}\label{5.11}
        \begin{cases}
        \frac{\mathrm{d}u_1}{\mathrm{d}t} = a_{10}u_1 + a_{01}v_1 + a_{20}u_1^2 + a_{11}u_1v_1, \\
        \frac{\mathrm{d}v_1}{\mathrm{d}t} = b_{10}u_1 + b_{01}v_1 + b_{20}u_1^2 + b_{11}u_1v_1 + b_{02}v_1^2 + O(|(u_1, v_1)|^3),
        \end{cases}
        \end{equation}
        where:

\[
        \begin{aligned}
            &a_{10} = \frac{1}{2} - 2h, \quad a_{01} = 2h - \frac{1}{2}, \quad a_{20} = -1, \quad a_{11} = -q, \\
            &b_{10} = s(2h-m), \quad b_{01} = s(m-2h), \quad b_{20} = \frac{s}{2h}(m-2h), \\
            &b_{11} = s\left(3 - \frac{m}{h}\right), \quad b_{02} = s\left(\frac{m}{2h} - 2\right).
        \end{aligned}
        \]

        Applying the transformation:

\[
        (u_1, v_1) = \left(u_2 + v_2, u_2 + \frac{b_{01}}{a_{01}}v_2\right),
        \]

        yields the system:
        \begin{equation}\label{5.12}
        \begin{cases}
        \frac{\mathrm{d}u_2}{\mathrm{d}t} = c_{20}u_2^2 + c_{11}u_2v_2 + c_{02}v_2^2 + O(|(u_2, v_2)|^3), \\
        \frac{\mathrm{d}v_2}{\mathrm{d}t} = (a_{10} + b_{01})v_2 + d_{20}u_2^2 + d_{11}u_2v_2 + d_{02}v_2^2 + O(|(u_2, v_2)|^3),
        \end{cases}
        \end{equation}
        where:

\[
        \begin{aligned}
            &c_{20} = \frac{s(m-2h)(1+q)}{a_{01}+b_{10}} \neq 0, \\
            &d_{01} = a_{10} + b_{01} = -(a_{01} + b_{10}) \neq 0.
        \end{aligned}
        \]

        By Lemma (\ref{l6}), \(E_7\) is a saddle-node.

        \item When \(h \neq \frac{1}{4}, m \neq 2h, s = s_1\), we have \(\mathrm{tr}(J(E_7)) = 0\). Under this condition, the Jacobian matrix of system (\ref{5.2}) at \(E_7\) has a double zero eigenvalue. Proceeding as in (1), we first shift the equilibrium point to the origin and expand the system at the origin, giving the form in (\ref{5.11}). Applying the transformation:

\[
        u_1 = u_3, \quad v_1 = u_3 + \frac{1}{a_{01}}v_3,
        \]

        yields the system:
        \begin{equation}\label{5.13}
        \begin{cases}
        \frac{\mathrm{d}u_3}{\mathrm{d}t} = v_3 + e_{20}u_3^2 + e_{11}u_3v_3 + O(|(u_3, v_3)|^3), \\
        \frac{\mathrm{d}v_3}{\mathrm{d}t} = f_{20}u_3^2 + f_{11}u_3v_3 + f_{02}v_3^2 + O(|(u_3, v_3)|^3),
        \end{cases}
        \end{equation}
        where:

\[
        \begin{aligned}
            &e_{20} = a_{20} + a_{11}, \quad e_{11} = \frac{a_{11}}{a_{01}}, \\
            &f_{20} = a_{01}(b_{20} + b_{11} + b_{02} - a_{20} - a_{11}), \\
            &f_{11} = b_{11} + 2b_{02} - a_{11}, \quad f_{02} = \frac{b_{02}}{a_{01}}.
        \end{aligned}
        \]

        By Lemma (\ref{l11}), system (\ref{5.13}) is equivalent near the origin to:

\[
        \begin{cases}
        \frac{\mathrm{d}u_3}{\mathrm{d}t} = v_3 + O(|(u_3, v_3)|^3), \\
        \frac{\mathrm{d}v_3}{\mathrm{d}t} = g_{20}u_3^2 + g_{11}u_3v_3 + O(|(u_3, v_3)|^3),
        \end{cases}
        \]

        where \(g_{20} = f_{20} = (2h - \frac{1}{2})(1+q) \neq 0\), and \(g_{11} = f_{11} + 2e_{20} = -2(1+q) < 0\). Therefore, when \(h \neq \frac{1}{4}, m \neq 2h, s = s_1\), \(E_7\) is a codimension-2 cusp point.
    \end{enumerate}
\end{proof}

\begin{theorem}
    When \(\Delta_{2} > 0\), system (\ref{5.2}) has an interior equilibrium point \(E_8 = (x_8, y_8) = \left(\frac{C+D}{2}, \frac{C+D}{2}\right)\). Define \(s_2 = \frac{2x_8 + qx_8 - 1}{m - x_8}\).

    \begin{enumerate}
        \item If \(m > x_8\), the equilibrium point \(E_8\) is a saddle point.
        \item If \(m < x_8\) and \(s > s_2\), the equilibrium point \(E_8\) is a stable node or focus.
        \item If \(m < x_8\) and \(s < s_2\), the equilibrium point \(E_8\) is an unstable node or focus.
        \item If \(m < x_8\) and \(s = s_2\), the equilibrium point \(E_8\) is a weak center.
    \end{enumerate}
\end{theorem}

\begin{proof}
    The Jacobian matrix of system (\ref{5.2}) at equilibrium point \(E_8\) is:

\[
    J(E_8) =
    \begin{pmatrix}
    -2x_8 - qy_8 + 1 & -qx_8 \\
    s(y_8-m) & s(m-y_8)
    \end{pmatrix}.
    \]

    The determinant of \(J(E_8)\) is:

\[
    \mathrm{det}(J(E_8)) = s(m - x_8)(-2x_8 - 2qx_8 + 1),
    \]

    and the trace is:

\[
    \mathrm{tr}(J(E_8)) = s(m - x_8) + (-2x_8 - qx_8 + 1).
    \]

    Note that \(-2x_8 - 2qx_8 + 1 = -2x_8(q+1) + 1 < 0\).

    \begin{enumerate}
        \item If \(m > x_8\), \(\mathrm{det}(J(E_8)) < 0\), so \(E_8\) is a saddle point.
        \item If \(m < x_8\) and \(s > s_2\), \(\mathrm{det}(J(E_8)) > 0\) and \(\mathrm{tr}(J(E_8)) < 0\), so \(E_8\) is a stable node or focus.
        \item If \(m < x_8\) and \(s < s_2\), \(\mathrm{det}(J(E_8)) > 0\) and \(\mathrm{tr}(J(E_8)) > 0\), so \(E_8\) is an unstable node or focus.
        \item If \(m < x_8\) and \(s = s_2\), \(\mathrm{det}(J(E_8)) > 0\) and \(\mathrm{tr}(J(E_8)) = 0\), so \(E_8\) is a weak center. The eigenvalues of \(J(E_8)\) are \(\lambda_{1,2} = \pm \sqrt{\mathrm{det}(J(E_8))}i\).
    \end{enumerate}
\end{proof}

\begin{theorem}
    When \(\Delta_{2} > 0\), system (\ref{5.2}) has an interior equilibrium point \(E_9 = (x_9, y_9) = \left(\frac{C-D}{2}, \frac{C-D}{2}\right)\). Define \(s_3 = \frac{2x_9 + qx_9 - 1}{m - x_9}\).

    \begin{enumerate}
        \item If \(m < x_9\), the equilibrium point \(E_9\) is a saddle point.
        \item If \(m > x_9\) and \(s < s_3\), the equilibrium point \(E_9\) is a stable node or focus.
        \item If \(m > x_9\) and \(s > s_3\), the equilibrium point \(E_9\) is an unstable node or focus.
        \item If \(m > x_9\) and \(s = s_3\), the equilibrium point \(E_9\) is a weak center.
    \end{enumerate}
\end{theorem}

\begin{proof}
    The Jacobian matrix of system (\ref{5.2}) at equilibrium point \(E_9\) is:

\[
    J(E_9) =
    \begin{pmatrix}
    -2x_9 - qy_9 + 1 & -qx_9 \\
    s(y_9-m) & s(m-y_9)
    \end{pmatrix}.
    \]

    The determinant of \(J(E_9)\) is:

\[
    \mathrm{det}(J(E_9)) = s(m - x_9)(-2x_9 - 2qx_9 + 1),
    \]

    and the trace is:

\[
    \mathrm{tr}(J(E_9)) = s(m - x_9) + (-2x_9 - qx_9 + 1).
    \]

    Note that \(-2x_9 - 2qx_9 + 1 = -2x_9(q+1) + 1 > 0\).

    \begin{enumerate}
        \item If \(m < x_9\), \(\mathrm{det}(J(E_9)) < 0\), so \(E_9\) is a saddle point.
        \item If \(m > x_9\) and \(s < s_3\), \(\mathrm{det}(J(E_9)) > 0\) and \(\mathrm{tr}(J(E_9)) < 0\), so \(E_9\) is a stable node or focus.
        \item If \(m > x_9\) and \(s > s_3\), \(\mathrm{det}(J(E_9)) > 0\) and \(\mathrm{tr}(J(E_9)) > 0\), so \(E_9\) is an unstable node or focus.
        \item If \(m > x_9\) and \(s = s_3\), \(\mathrm{det}(J(E_9)) > 0\) and \(\mathrm{tr}(J(E_9)) = 0\), so \(E_9\) is a weak center. The eigenvalues of \(J(E_9)\) are \(\lambda_{1,2} = \pm \sqrt{\mathrm{det}(J(E_9))}i\).
    \end{enumerate}
\end{proof}

\section{Bifurcation analysis}

Based on the discussion of the existence and stability of equilibrium points for system (\ref{5.2}) in Section (2.1), this section focuses on investigating the various branches that system (\ref{5.2}) may exhibit under different parameter conditions, including saddle-node bifurcations, Hopf bifurcations, and codimension-two Bogdanov-Takens bifurcations. By analyzing these branching behaviors, not only can the dynamic evolution patterns of the system be revealed more clearly, but its potential dynamical characteristics can also be further explored.

\subsection{Saddle-node bifurcation}

According to Theorem 1, when $h<\frac{1}{4}$, the system (\ref{5.2}) has two boundary equilibrium points $E_{2}$ and $E_{3}$. Let $h_2=\frac{1}{4}$. When $h=h_2=\frac{1}{4}$, the system has one boundary equilibrium point $E_{1}$, and $E_{1}$ is a saddle-node. When $h>\frac{1}{4}$, the system has no boundary equilibrium points. As the parameter $h$ varies near $h_2$, the number of boundary equilibrium points in the system changes, which may lead to a saddle-node bifurcation.

\begin{lemma}\label{l7}\cite{030}
    Consider the system (\ref{2.3}), where $f(x_{0},y_{0},u_{0})=0$, and its Jacobian matrix $J=Df(x_0,y_0,\mu_0)$ has a zero eigenvalue $\lambda=0$. The eigenvector corresponding to $\lambda=0$ is $v = (v_1, v_2)^T$, while the real parts of all other eigenvalues are nonzero. Similarly, the eigenvector corresponding to $\lambda=0$ for $J^T$ is $w = (w_1, w_2)^T$. When the following conditions are satisfied and $\mu$ crosses the critical value $\mu=\mu_{0}$, the system undergoes a saddle-node bifurcation at the equilibrium point $(x_{0},y_{0})$.
    $$w^\mathrm{T}f_\mu(x_0,y_0,\mu_0)\neq0,\quad w^\mathrm{T}[D^2f(x_0,y_0,\mu_0)(v,v)]\neq0,$$
    \noindent where $D^{2}f(x_{0},y_{0},\mu_{0})(v,v)=\frac{\partial^{2}f(x_{0},y_{0},\mu_{0})}{\partial x^{2}}v_{1}^{2}+2\frac{\partial^{2}f(x_{0},y_{0},\mu_{0})}{\partial x\partial y}v_{1}v_{2}+\frac{\partial^{2}f(x_{0},y_{0},\mu_{0})}{\partial y^{2}}v_{2}^{2}$.
\end{lemma}

\begin{theorem}
   Choose parameter $h$ as the bifurcation parameter, then system \ref{5.2} undergoes a saddle-node bifurcation at the equilibrium point $E_{1}$, with the critical bifurcation parameter being $h_2=\frac{1}{4}$.
\end{theorem}

\begin{proof}
  The Jacobian matrix of system \ref{5.2} at $E_1$ is given by:
  $$J\left(E_1\right)=\begin{pmatrix}0&-\frac{q}{2}\\0&-sm\end{pmatrix}$$

  The eigenvectors corresponding to the zero eigenvalue of the matrices $(J(E_1))$ and $(J(E_1)^T)$ are:
  $$v=\begin{pmatrix}v_1\\v_2\end{pmatrix}=\begin{pmatrix}1\\0\end{pmatrix}\quad,\quad w=\begin{pmatrix}w_1\\w_2\end{pmatrix}=\begin{pmatrix}2sm\\-q\end{pmatrix}.$$
  \noindent Let
  $$f\left(x,y,h\right)=\begin{pmatrix}\dot{x}\\\dot{y}\end{pmatrix}=\begin{pmatrix}f_1\\f_2\end{pmatrix}=\begin{pmatrix}x\left(1-x\right)-qxy-h\\\\sy\left(1-\frac{y}{x}\right)(y-m)\end{pmatrix}$$
  \noindent Then $f_h(E_1,h_2)=\begin{pmatrix}-1\\0\end{pmatrix},$
  $$ D^2f\left(E_1,h_2\right)(v,v)=\begin{pmatrix}\frac{\partial^2f_1}{\partial x^2}v_1^2+2\frac{\partial^2f_1}{\partial x\partial y}v_1v_2+\frac{\partial^2f_1}{\partial y^2}v_2^2\\\frac{\partial^2f_2}{\partial x^2}v_1^2+2\frac{\partial^2f_2}{\partial x\partial y}v_1v_2+\frac{\partial^2f_2}{\partial y^2}v_2^2\end{pmatrix}=\begin{pmatrix}-2\\0\end{pmatrix}$$
  \noindent Hence, $w^T f_h(E_1,h_2) = -2sm \neq 0$ and $w^T \left[D^2f(E_1,h_2)(v,v)\right] = -4sm\neq 0$.
  Thus, with critical bifurcation parameter $h_2$, $v$ and $w$ satisfy the transversality condition for a saddle-node bifurcation at the equilibrium point $E_1$. By Lemma (\ref{l7}), system (\ref{5.2}) undergoes a saddle-node bifurcation at $E_1$.
\end{proof}

  When the system parameter transitions from one side of $h_2$ to the other, the number of boundary equilibrium points in system (\ref{5.2}) changes from zero to two. For system (\ref{5.2}) with only prey populations, if the parameter satisfies $h > \frac{1}{4}$, the prey population inevitably goes extinct. Within the parameter range $0 < h < \frac{1}{4}$, by appropriately selecting initial conditions, the prey population can be preserved, maintaining its survival state.

  Similar saddle-node bifurcation analyses can be conducted for the equilibrium points $E_4$ and $E_7$. These are not discussed in detail here.

\subsection{Hopf bifurcation}

In this section, we consider the Hopf bifurcation. According to Theorem 13, when $\Delta_{2}>0$, the system (\ref{5.2}) has an interior equilibrium point $E_8=(x_8,y_8)$. Let $s_2=\frac{2x_8+qx_8-1}{m-x_8}$. When $m<x_8,\quad s>s_2$, the equilibrium point $E_8$ is a stable node or focus; when $m<x_8,\quad s<s_2$, the equilibrium point $E_8$ is an unstable node or focus; and when $m<x_8,\quad s=s_2$, the equilibrium point $E_8$ is a weak center. Next, we investigate whether a Hopf bifurcation occurs near the equilibrium point $E_8$ as the parameter $s$ varies within a small neighborhood of $s_2$.

\begin{lemma}\label{l9}\cite{029}
    For a general planar system:
    $$\begin{cases}\frac{\mathrm{d}x}{\mathrm{d}t}=ax+by+p(x,y),\\\frac{\mathrm{d}y}{\mathrm{d}t}=cx+dy+q(x,y),&\end{cases}$$  
    \noindent where $\Delta=ad-bc>0, a+d=0$, $p(x,y)=\sum_{i+j=2}^{\infty}a_{ij}x^{i}y^{j}$, and $q(x,y)=\sum_{i+j=2}^{\infty}b_{ij}x^{i}y^{j}$, with $i\geq0$, $j\geq0$, and both $p(x,y)$ and $q(x,y)$ being convergent series. The system possesses a pair of purely imaginary eigenvalues, and the origin is a weak center. The first Liapunov coefficient $\sigma$ for the Hopf bifurcation can be calculated using the following formula:
    $$\begin{aligned}\sigma&=\frac{-3\pi}{2b\Delta^{3/2}}\{[ac(a_{11}^{2}+a_{11}b_{02}+a_{02}b_{11})+ab(b_{11}^{2}+a_{20}b_{11}+a_{11}b_{02})+c^{2}(a_{11}a_{02}+2a_{02}b_{02})\\&-2ac(b_{02}^{2}-a_{20}a_{02})-2ab(a_{20}^{2}-b_{20}b_{02})-b^{2}(2a_{20}b_{20}+b_{11}b_{20})+(bc-2a^{2})(b_{11}b_{02}-a_{11}a_{20})]\\&-(a^{2}+bc)[3(cb_{03}-ba_{30})+2a(a_{21}+b_{12})+(ca_{12}-bb_{21})]\}.\end{aligned}$$
\end{lemma}

\begin{theorem}
  When $\Delta_{2}>0$ and $m<x_8$, choosing parameter $s$ as the bifurcation parameter, as parameter $s$ varies within a small neighborhood of $s_2$, system (\ref{5.2}) undergoes a Hopf bifurcation near the equilibrium point $E_8$.
\end{theorem}
\begin{proof}
  The trace of the Jacobian matrix at the equilibrium point $E_8$ is given by $\mathrm{tr}\left(J(E_8)\right)=s(m-x_8)+(-2x_8-qx_8+1)$. By differentiating the trace with respect to parameter $s$ at $s_2$, we have:
  $$\frac{\mathrm{dtr}(J(E_8))}{\mathrm{d}s}|_{s=s_2}=m-x_8<0,\quad\mathrm{tr}(J(E_8))_{s_2}=0$$
  \noindent Therefore, parameter $s$ satisfies the conditions for a Hopf bifurcation. Choosing $s$ as the bifurcation parameter, system (\ref{5.2}) undergoes a Hopf bifurcation near the equilibrium point $E_8$. 
\end{proof}

To determine the direction of the Hopf bifurcation, we calculate the first Lyapunov coefficient. First, the equilibrium point $E_8$ is shifted to the origin via the coordinate transformation $$(x,y)=(u_1+x_8,v_1+y_8),$$
\noindent and Taylor expansion at the origin gives:
\begin{equation}\label{5.14}
  \begin{cases}\frac{\mathrm{d}u_{1}}{\mathrm{d}t}=a_{10}u_{1}+a_{01}v_{1}+a_{20}u_{1}^{2}+a_{11}u_{1}v_{1}+O(|(u_{1},v_{1})|^{4}),\\\frac{\mathrm{d}v_{1}}{\mathrm{d}t}=b_{10}u_{1}+b_{01}v_{1}+b_{20}u_{1}^{2}+b_{11}u_{1}v_{1}+b_{02}v_{1}^{2}+b_{30}u_{1}^{3}+b_{21}u_{1}^{2}v_{1}\\+b_{12}u_{1}v_{1}^{2}+b_{03}v_{1}^3+O(|(u_{1},v_{1})|^{4}),&\end{cases}
\end{equation}
\noindent where $a_{10}=-2x_8-qx_8+1,\quad a_{01}=-qx_8,\quad a_{20}=-1,\quad a_{11}=-q,\quad b_{10}=s(x_8-m),\quad b_{01}=s(m-x_8),\quad b_{20}=s(\frac{m}{x_8}-2),\quad b_{11}=s(3-\frac{2m}{x_8}),\quad b_{02}=s(\frac{m}{x_8}-2),\quad b_{30}=s(\frac{1}{x_8}-\frac{m}{x_{8}^2}),\quad b_{21}=s(\frac{2m}{x_8^2}-\frac{3}{x_8}),\quad b_{12}=s(\frac{3}{x_8}-\frac{m}{x_8^2}),\quad b_{03}=-\frac{s}{x_8}$.

According to Lemma (\ref{l9}), the expression for the first Lyapunov coefficient $\sigma$ of the system is:
 $$\sigma=\frac{-3\pi}{2a_{01}M^{3/2}}\sum_{i=1}^8\varphi_i,$$
 \noindent where $$\begin{aligned}
&M=a_{10}b_{01}-a_{01}b_{10},\\&\varphi_{1}=a_{10}b_{10}\left(a_{11}^{2}+a_{11}b_{02}\right), \\ &\varphi_{2}=a_{10}a_{01}\left(b_{11}^{2}+a_{20}b_{11}+a_{11}b_{02}\right),\\&\varphi_{3}=0,\\&\varphi_{4}=-2a_{10}b_{10}b_{02}^{2},\\&\varphi_{5}=-2a_{10}a_{01}\left(a_{20}^{2}-b_{20}b_{02}\right)\\
&\varphi_{6}=-a_{01}^{2}\left(2a_{20}b_{20}+b_{11}a_{20}\right),\\&\varphi_{7}=\left(a_{01}b_{01}-2a_{10}^{2}\right)\left(b_{11}b_{02}-a_{11}a_{20}\right),\\&\varphi_{8}=-\left(a_{10}^{2}+a_{01}b_{10}\right)\left[3(b_{10}b_{03}-a_{01}a_{30})+2a_{10}b_{12}-a_{01}b_{21}\right].
\end{aligned}
$$

  When $\sigma<0$, system (\ref{5.2}) undergoes a supercritical Hopf bifurcation, with a stable limit cycle near the equilibrium point $E_8$. When $\sigma>0$, system (\ref{5.2}) undergoes a subcritical Hopf bifurcation, with an unstable limit cycle near $E_8$.

  When parameter $s$ is chosen as the bifurcation parameter, system (\ref{5.2}) undergoes a Hopf bifurcation at equilibrium point $E_8$. Under specific parameter conditions, the system generates a limit cycle near $E_8$, indicating that the predator-prey relationship may exhibit periodic fluctuations. These periodic oscillations pose challenges to ecosystem management, potentially impacting resource utilization efficiency and conservation measures, thus increasing management complexity. However, by studying the dynamics near $E_8$, we can formulate reasonable resource utilization strategies and conservation measures to effectively address the system's periodic changes, thereby promoting sustainable development of the ecosystem.

  Similar Hopf bifurcation analyses can be conducted for the equilibrium point $E_9$. These details are omitted here.

\subsection{Bogdanov-Takens bifurcation}

When $\Delta_{2}=0$, system (\ref{5.2}) has an interior equilibrium point $E_7=(x_7,y_7)=(2h,2h)$. When $h\ne\frac{1}{4}$, $m\ne2h$, and $s= s_1$, the equilibrium point $E_7$ is a codimension-two cusp point. Next, we will analyze whether a codimension-two Bogdanov-Takens bifurcation occurs near the interior equilibrium point $E_7$. Since $\Delta_{2}=0$, i.e., $h=\frac{1}{4(q+1)}$, let $h_3=\frac{1}{4(q+1)}$.

\begin{theorem}
  For system (\ref{5.2}), if $s$ and $h$ are selected as bifurcation parameters, then as parameters $s$ and $h$ vary within a small neighborhood of $s_1$ and $h_3$, the system undergoes a codimension-two Bogdanov-Takens bifurcation near the equilibrium point $E_8$.
\end{theorem}

\begin{proof}
    Introduce small perturbations to parameters $s_1$ and $h_3$ at the codimension-two cusp point $E_7$ of system (\ref{5.2}): $(h,s)=(h_3+\eta_1,s_1+\eta_2)$. The perturbed system becomes:
\begin{equation}\label{5.15}
  \begin{cases}
    \frac{\mathrm{d}x}{\mathrm{d}t}=x\left(1-x\right)-qxy-(h_3+\eta_1),\\
    \frac{\mathrm{d}y}{\mathrm{d}t}=(s_1+\eta_2)y\left(1-\frac{y}{x}\right)(y-m),
  \end{cases}
\end{equation}
where $(\eta_1,\eta_2)$ is a parameter vector within a small neighborhood of the origin. Clearly, when $\eta_1=\eta_2=0$, system (\ref{5.2}) has a codimension-two cusp point.

Using the coordinate transformation $(x,y)=(u_1+x_5,v_1+y_5)$, the equilibrium point $E_7$ is shifted to the origin. Expanding at the origin yields:
\begin{equation}\label{5.16}
  \begin{cases}
    \frac{\mathrm{d}u_1}{\mathrm{d}t}=a_{00}+a_{10}u_1+a_{01}v_1+a_{20}u_1^2+a_{11}u_1v_1+O(|(u_1,v_1)|^3),\\
    \frac{\mathrm{d}v_1}{\mathrm{d}t}=b_{10}u_1+b_{01}v_1+b_{20}u_1^2+b_{11}u_1v_1+b_{02}v_1^2+O(|(u_1,v_1)|^3).
  \end{cases}
\end{equation}

The coefficients are expressed as:
$$\begin{aligned}
  &a_{00}=-\eta_1,\quad a_{10}=\frac{1}{2}-2h_3,\quad a_{01}=2h_3-\frac{1}{2},\quad a_{20}=-1,\quad a_{11}=-q,\\
  &b_{10}=(s_1+\eta_2)(2h_3-m),\quad b_{01}=(s_1+\eta_2)(m-2h_3),\\
  &b_{20}=\frac{s_1+\eta_2}{2h}(m-2h_3),\quad b_{11}=(s_1+\eta_2)\left(3-\frac{m}{h_3}\right),\\
  &b_{02}=(s_1+\eta_2)\left(\frac{m}{2h_3}-2\right).
\end{aligned}$$

Through the transformation:
$$\begin{aligned}
  u_2 &= u_1,\\
  v_2 &= a_{10}u_1 + a_{01}v_1,
\end{aligned}$$
the system becomes:
\begin{equation}\label{5.17}
  \begin{cases}
    \frac{\mathrm{d}u_2}{\mathrm{d}t}=c_{00}+v_2+c_{20}u_2^2+c_{11}u_2v_2+O(|(u_2,v_2)|^3),\\
    \frac{\mathrm{d}v_2}{\mathrm{d}t}=d_{00}+d_{10}u_2+d_{01}v_2+d_{20}u_2^2+d_{11}u_2v_2+d_{02}v_2^2+O(|(u_2,v_2)|^3).
  \end{cases}
\end{equation}

The coefficients are expressed as:
$$\begin{aligned}
  &c_{00}=a_{00},\quad c_{20}=a_{20}-\frac{a_{11}a_{10}}{a_{01}},\quad c_{11}=\frac{a_{11}}{a_{01}},\\
  &d_{00}=a_{00}a_{10},\quad d_{10}=a_{01}b_{10}-a_{10}b_{01},\quad d_{01}=a_{10}+b_{01},\\
  &d_{20}=a_{10}a_{20}+a_{01}b_{20}-a_{10}b_{11}-\frac{a_{10}^2a_{11}}{a_{01}}+\frac{a_{10}^2b_{02}}{a_{01}},\\
  &d_{11}=b_{11}+\frac{a_{10}a_{11}}{a_{01}}-\frac{2a_{10}b_{02}}{a_{01}},\quad d_{02}=\frac{b_{02}}{a_{01}}.
\end{aligned}$$
\noindent Perform the following transformation on system (\ref{5.17}):
$$
\begin{cases}
u_3=u_2, \\
v_3=c_{00}+v_2+c_{20}u_2^2+c_{11}u_2v_2+O(\mid(u_2,v_2)\mid^3),
\end{cases}
$$
resulting in the system:
\begin{equation}\label{5.18}
\begin{cases}
\frac{\mathrm{d}u_3}{\mathrm{d}t}=v_3, \\
\frac{\mathrm{d}v_3}{\mathrm{d}t}=e_{00}+e_{10}u_3+e_{01}v_3+e_{20}u_3^2+e_{11}u_3v_3+e_{02}v_3^2+O(\mid(u_3,v_3)\mid^3).
\end{cases}
\end{equation}

\noindent where:
$$
\begin{aligned}
&e_{00}=d_{00}-c_{00}d_{01}+c_{00}^2d_{02}, \\
&e_{10}=d_{10}+c_{11}d_{00}-c_{00}d_{11}+c_{00}^2c_{11}d_{02}, \\
&e_{01}=d_{01}-c_{00}c_{11}-2c_{00}d_{02}, \\
&e_{20}=d_{20}+c_{11}d_{10}-c_{20}d_{10}+2c_{00}c_{20}d_{02}, \\
&e_{11}=d_{11}+2c_{20}-c_{00}c_{11}^2-2c_{00}c_{11}d_{02}, \\
&e_{02}=d_{02}+c_{11}.
\end{aligned}
$$

\noindent Apply the time transformation $\mathrm{d}t=(1-e_{02}u_3)\mathrm{d}\tau$ to system (\ref{5.18}), yielding:
\begin{equation}\label{5.19}
\begin{cases}
\frac{\mathrm{d}u_3}{\mathrm{d}\tau}=v_3(1-e_{02}u_3), \\
\frac{\mathrm{d}v_3}{\mathrm{d}\tau}=(1-e_{02}u_3)\big(e_{00}+e_{10}u_3+e_{01}v_3+e_{20}u_3^2+e_{11}u_3v_3+e_{02}v_3^2+O(\mid(u_3,v_3)\mid^3)\big).
\end{cases}
\end{equation}
\noindent Apply the transformation $(u_4, v_4) = (u_3, v_3(1-e_{02}u_3))$ to system (\ref{5.19}), yielding:
\begin{equation}\label{5.20}
\begin{cases}
\frac{\mathrm{d}u_4}{\mathrm{d}\tau}=v_4, \\
\frac{\mathrm{d}v_4}{\mathrm{d}\tau}=f_{00}+f_{10}u_4+f_{01}v_4+f_{20}u_4^2+f_{11}u_4v_4+O(\mid(u_4,v_4)\mid^3),
\end{cases}
\end{equation}

\noindent where $f_{00}=e_{00},\quad f_{10}=e_{10}-2e_{00}e_{02},\quad f_{01}=e_{01},\quad f_{20}=e_{20}-2e_{02}e_{10}+e_{00}e_{02}^2,\quad f_{11}=-e_{01}e_{02}+e_{11}$.

\noindent Since $h_3 \neq \frac{1}{4}$, we have:
$$
\lim_{(\eta_1, \eta_2) \to (0,0)}f_{20} = \left(\frac{1}{2}-2h_3\right)(q+1) \neq 0.
$$

\noindent Assume $\lim_{(\eta_1, \eta_2) \to (0,0)}f_{20} < 0$, then for sufficiently small $\eta_1$ and $\eta_2$, $f_{20}(\eta) < 0$. When $\lim_{(\eta_1, \eta_2) \to (0,0)}f_{20} > 0$, a similar method can be applied. Perform the following transformation on system (\ref{5.20}):
$$
(u_5, v_5) = \left(u_4, \frac{v_4}{\sqrt{-f_{20}}}\right), \quad t_1 = \sqrt{-f_{20}}\tau,
$$
resulting in the system:
\begin{equation}\label{5.21}
\begin{cases}
\frac{\mathrm{d}u_5}{\mathrm{d}t_1}=v_5, \\
\frac{\mathrm{d}v_5}{\mathrm{d}t_1}=g_{00}+g_{10}u_5+g_{01}v_5-u_5^2+g_{11}u_5v_5+O(\mid(u_5,v_5)\mid^3),
\end{cases}
\end{equation}

\noindent where $g_{00} = -\frac{f_{00}}{f_{20}}, \quad g_{10} = -\frac{f_{10}}{f_{20}}, \quad g_{01} = \frac{f_{01}}{\sqrt{-f_{20}}}, \quad g_{11} = \frac{f_{11}}{\sqrt{-f_{20}}}$.

\noindent Perform the transformation $(u_6, v_6) = \left(u_5-\frac{g_{10}}{2}, v_5\right)$ on system (\ref{5.21}), yielding:
\begin{equation}\label{5.22}
\begin{cases}
\frac{\mathrm{d}u_6}{\mathrm{d}t_1}=v_6, \\
\frac{\mathrm{d}v_6}{\mathrm{d}t_1}=h_{00}+h_{01}v_6-u_6^2+h_{11}u_6v_6+O(\mid(u_6,v_6)\mid^3),
\end{cases}
\end{equation}

\noindent where $h_{00} = g_{00}+\frac{1}{4}g_{10}^2, \quad h_{01} = g_{01}+\frac{1}{2}g_{10}g_{11}, \quad h_{11} = g_{11}$.

\noindent Since $\lim_{(\eta_1, \eta_2) \to (0,0)}h_{11} = -(s_1+2+q) < 0$, for sufficiently small $\eta_1$ and $\eta_2$, $h_{11}(\eta) < 0$. Perform the transformation:
$$
(u_7, v_7) = \left(h_{11}^2u_6, -h_{11}^3v_6\right), \quad t_2 = -\frac{1}{h_{11}}t_1,
$$
resulting in the system:
\begin{equation}\label{5.23}
\begin{cases}
\frac{\mathrm{d}u_7}{\mathrm{d}t_2}=v_7, \\
\frac{\mathrm{d}v_7}{\mathrm{d}t_2}=l_{00}+l_{01}v_7+u_7^2+u_7v_7+O(\mid(u_7,v_7)\mid^3),
\end{cases}
\end{equation}

\noindent where $l_{00} = -h_{00}h_{11}^4, \quad l_{01} = -h_{01}h_{11}$.

\noindent Note:
$$
\left|\frac{\partial(l_{00}, l_{01})}{\partial(\eta_1, \eta_2)}\right|_{\eta_1=\eta_2=0} \neq 0.
$$

\noindent Therefore, for system (\ref{5.2}), by selecting $s$ and $h$ as bifurcation parameters, when parameters $s, h$ vary in a small neighborhood near $s_1, h_3$, the system exhibits a codimension-2 Bogdanov-Takens bifurcation near the equilibrium point $E_8$.
\end{proof}

\section{Conclusion}

The Holling I type Leslie-Gower model, where predators exhibit the Allee effect and implement constant capture on prey, has three types of equilibrium points under different parameters: $y=0$, $y=m$, and $y=x$. These equilibrium points may be stable or unstable nodes (foci), saddle points, weak centers, or cusp points under varying parameters. We selected three representative equilibrium points and analyzed bifurcation conditions near these equilibrium points as the parameters changed. By choosing the parameter $h$ as the bifurcation parameter, the system undergoes a saddle-node bifurcation at the equilibrium point $E_{1}$, with the critical bifurcation parameter being $h_2=\frac{1}{4}$. By choosing the parameter $s$ as the bifurcation parameter, when $s$ varies within a small neighborhood near $s_2$, the system undergoes a Hopf bifurcation near the equilibrium point $E_8$. The first Lyapunov coefficient \(\sigma<0\) corresponds to a supercritical Hopf bifurcation, while \(\sigma>0\) corresponds to a subcritical Hopf bifurcation. When both parameters $s$ and $h$ are selected as bifurcation parameters, and they vary within a small neighborhood around $s_1$ and $h_3$, the system undergoes a codimension-2 Bogdanov-Takens bifurcation near the equilibrium point $E_8$.
\newpage
\bibliographystyle{plainnat}
\bibliography{template}


\end{document}